\documentclass[reqno,10pt]{amsart}
\numberwithin{equation}{section}

\usepackage{amssymb,amsmath,latexsym}
\usepackage{amsmath, amsthm, amscd, amsfonts, amssymb, graphicx, color}
\usepackage[colorlinks=true,linkcolor=red,citecolor=blue]{hyperref}
\usepackage{marvosym}
\usepackage{epsfig, epstopdf}
\usepackage{cite}
\usepackage{float}
\title[   ]{ A strong convergence theorem for solving an equilibrium
problem and a fixed point problem using the Bregman distance}
 \author[Ghadampour, Soori, Agarwal and O$'$Regan]{Mostafa Ghadampour$^{1}$, Ebrahim Soori$^{*,2}$, Ravi P. Agarwal$^{3}$,  Donal O$'$Regan$^{4}$  $\,\,$ }
 \thanks{ \!\! \!\! \!\! \!\!* Corresponding author  
  \\ E-mail addresses:  m.ghadampour@gmail.com(m. ghadampour), sori.e@lu.ac.ir (E. Soori),  agarwal@tamuk.edu(R. P. Agarwal), donal.oregan@nuigalway.ie(D. O$'$Regan). }

\theoremstyle{plain}
\newtheorem{prop}{\textbf{Proposition}}

\newtheorem{lem}{\textbf{Lemma}}[section]
\newtheorem{thm}[lem]{\textbf{Theorem}}

\newtheorem{ex}[lem]{\textbf{Example}}

\newtheorem{re}{\textbf{Remark}}
\theoremstyle{definition}

\theoremstyle{definition}
\theoremstyle{remark}

\begin{document}
\maketitle
\linespread{1}
\begin{center}
\begin{normalsize}
   $^{1}$ Department of Mathematics, Payame Noor University
P. O. Box 19395-3697, Tehran, Iran.\\
$^2$ Department of Mathematics, Lorestan University, Lorestan, Khoramabad,
Iran.\\
   $^{3}$ Department of Mathematics
 Texas A$\&$M University-Kingsville 700 University Blvd., MSC 172 Kingsville, Texas, USA.\\
 $^{4}$ School of Mathematics, Statistics, National University of Ireland, Galway, Ireland.\\
 \end{normalsize}
  \end{center}
\begin{abstract}
\begin{normalsize}
In this paper, using the Bregman distance,  we introduce a new projection-type algorithm for finding a common element of the set of solutions of an equilibrium problem and the set of fixed points. Then the strong convergence of the sequence generated by the algorithm will be established under  suitable conditions. Finally, using MATLAB software,  we present a numerical example to illustrate the convergence performance  of our algorithm.
\end{normalsize}
\end{abstract}
\begin{normalsize}
   \end{normalsize}

\section{ Introduction}
Let $C$ be a nonempty closed and convex subset of a reflexive real Banach space X. Throughout this paper, $X^*$ denotes the dual space of $X$. The norm and the duality pairing between
$X$ and $X^*$ are denoted by $\| . \|$ and $\langle ., .\rangle$, respectively. Now $\mathbb{R}$ stands for the set of real numbers. The equilibrium problem for a bifunction $g : C\times C \rightarrow \mathbb{R}$ satisfying the condition
$g(x, x) = 0$ for every $x \in C$ is stated as follows
 \begin{equation}\label{EP0}
     \text{Find}\;  y^*\in C\; \; \text{such}\; \text{that} \;\; g(x, y^*)\leq 0,
   \end{equation}
   for all $x\in C$. The set of solutions of \eqref{EP0} is denoted by $EP(g)$.

It is well known that variational inequalities arise in optimal control problems, optimization problems, fixed point problems, partial differential equations, engineering and equilibrium models and hence they have been formulated by many authors in recent years(see \cite{zhaox}, \cite{tdvdvt} ).

Tada and Takahashi \cite{tada} proposed the hybrid method for finding the
common element of the set of solutions of the monotone equilibrium problem \eqref{EP0} and a set of fixed points of a nonexpansive map represented in their algorithm.
Anh \cite{Anh} proposed the hybrid extragradient iteration method for finding a common element of the set of fixed points of a nonexpansive
mapping and the set of solutions of equilibrium problems for a pseudomonotone and Lipschitz-type continuous bifunction.
Eskandani et al.\cite{egrm} proposed a hybrid extragradient method and they introduced a new iterative process for approximating a common element of the set of solutions of equilibrium problems involving pseudomonotone bifunctions and the set of common fixed points of a finite family of multi-valued Bregman relatively nonexpansive mappings in Banach spaces. They proved that for any $x\in C$ the mapping $y \rightarrow
g(x, y) + D_f (x, y)$ has a unique minimizer where $g(x, .)$ is proper,  convex, lower semicontinuous and $D_f$ is the Bregman distance.
Also, Jolaoso et al. \cite{jlot} proved that under some suitable conditions, a point $x^*\in EP(g)$
if and only if $x^*$ solves the following minimization problem:
\begin{equation*}
 \min\{\lambda g(x, y) + D_f (x, y): y\in C\}.
\end{equation*}

In this paper, motivated by the work of Jolaoso et al. \cite{jlot}, we will present a new projection-type algorithm for approximating a common solution of a Bregman nonexpansive mapping which is a solution of \eqref{EP0}  in the setting of reflexive Banach spaces. Then using MATLAB software, the
main result will be illustrated with a numerical example.
\section{preliminaries}
We present some preliminaries and
lemmas which will be used in the next section. Let $f : X \rightarrow (-\infty, \infty]$ be an admissible function, i.e., a proper, convex and lower semicontinuous function. The domain of $f$ is the set $\{x \in X : f(x) < \infty\}$ denoted
by $\text{dom} f$. The set of minimizers of f is denoted by Argmin $f$ and its unique element by $argmin_{x\in X} f(x)$, if Argmin $f$ is a singleton.
Let $x \in \text{int}$ $\text{dom} f $, for any $y \in X$, the directional derivative of f at $x$ is defined by
\begin{equation}\label{fcir}
  f^{\circ}(x, y) = \displaystyle\lim_{t\rightarrow 0}\frac{f(x + ty) - f (x)}{t}.
\end{equation}
If the limit \eqref{fcir} as $t\rightarrow 0$ exists for each $y$, then $f$ is
said to be G\^{a}teaux differentiable at $x$. The function $f$ is said to be G\^{a}teaux differentiable if it is G\^{a}teaux differentiable for all $x \in \text{int}$ $\text{dom} f$. When the limit as $t \rightarrow 0$ in \eqref{fcir} is attained uniformly for any $y \in X$ with $\|y\| = 1$, we say that $f$ is Fr\'{e}chet differentiable at $x$. Finally, $f$ is said to be uniformly Fr\'{e}chet differentiable on a subset $C$ of $X$ if
the limit is attained uniformly at each $x \in C$ and $\| y \| = 1$.
In this case, the gradient of $f$ at $x$ is the linear
function $\nabla f(x)$ which is defined by $\langle y, \nabla f(x)\rangle := f^{\circ}(x, y)$ for all $y\in X$.

Let $x\in \text{int}$ $\text{dom} f$. The subdifferential of $f$ at $x$ is the convex set defined by
\begin{equation}\label{rond}
  \partial f(x) = \{l\in X^*: f(x) +\langle y-x, l\rangle \leq f(y), \;\forall y\in X\},
\end{equation}
where the Fenchel conjugate of $f$  is the convex function $f^*: X^* \rightarrow (-\infty, \infty]$ defined by
\begin{equation*}
  f^*(l)=\sup\{\langle l, x\rangle - f(x) :x\in X\}.
\end{equation*}
It is well known that by the Young-Fenchel inequality, if $\partial f(x)$ is nonempty, then we have
\begin{equation*}
  \langle l, x\rangle\leq f(x) + f^*(l), \;\; \forall x\in X, \; l\in X^*,
\end{equation*}
and also
\begin{equation*}
  f(x) + f^*(l) = \langle l, x\rangle \Leftrightarrow l\in \partial f(x).
\end{equation*}

Let $X$ be a reflexive Banach space. The function $f$ is Legendre if and only if
it satisfies the following two conditions:
\begin{enumerate}
  \item [(L1)] $\text{int}$ $\text{dom} f\neq \emptyset$, $f$ is G\^{a}teaux differentiable and dom $\nabla f=$ int dom $f$.
  \item [(L2)] $\text{int}$ $\text{dom} f^*\neq \emptyset$, $f^*$ is G\^{a}teaux differentiable and dom $\nabla f^*=$ int dom $f^*$.
\end{enumerate}
Since $X$ is a reflexive Banach space, $(\partial f)^{-1}=\partial f^*$ (see \cite[p 83]{bjfsa}). Also, we know that $\nabla f = (\nabla f^*)^{-1}$, this together with conditions
(L1) and (L2) imply that ran$\nabla f=$ dom $\nabla f^*=$ int dom $f^*$ and ran$\nabla f^*=$ dom $\nabla f=$ int dom $f$. In addition, if $X$ is reflexive, then $f$  is Legendre if and only if $f^*$ is Legendre (see \cite[corollary 5.5]{bhhbjmcpl}).

   Let $f : X \rightarrow (-\infty, \infty]$ be a G\^{a}teaux differentiable function. The
bifunction $D_f : dom f \times int\; dom f \rightarrow [0,+\infty]$ defined by
\begin{equation}\label{Df}
D_f (y, x) := f (y) - f (x) - \langle y - x, \nabla f(x) \rangle,
\end{equation}
is called the Bregman distance with respect to f (see \cite{rdrsdpa}).
The Bregman distance does not satisfy the well known properties of a metric. Clearly, $D_f (x, x) = 0$, but $D_f (y, x) = 0$ may not imply $x = y$, but when $f$ is Legendre this indeed holds (see \cite[Theorem 7.3(vi)]{bhhbjmcpl}).

The modulus of total convexity at $x \in \text{int} \;\text{dom}f$ is the function $\upsilon_f (x, .) :
[0,\infty) \rightarrow [0,\infty]$, defined by
\begin{equation*}
  \upsilon_f (x, t) :=\inf\{ D_f(y,x) : y\in \text{dom}f, \;\|y - x\|=t\}.
\end{equation*}
The function $f$ is called totally convex at $x \in \text{int} \;\text{dom} f$ if $\upsilon_f (x, t)$ is positive
for any $t > 0$ \cite{bdian}. Let $C$ be a nonempty subset of $X$. The modulus of total convexity of $f$ on
$C$ is defined by
\begin{equation*}
  \upsilon_f (C, t) =\inf\{ \upsilon_f (x, t) : x\in C\cap\text{int}\; \text{dom}f\}.
\end{equation*}
The function $f$ is called totally convex on bounded subsets if $\upsilon_f (C, t)$ is
positive for any nonempty and bounded subset $C$ and for any $t > 0$.

The following result establishes the characteristic continuity properties
for the derivative of a lower semicontinuous convex function.
\begin{prop}\label{singdf}\cite{bdian}
Let $f$ be a lower semicontinuous convex function with $\text{int}$ $\text{dom} f \neq \emptyset$. Then the function $f$ is differentiable at the point $x\in \text{int}$ $\text{dom} f$ if and only if
$\partial f(x)$ consists of a single element.
\end{prop}
\begin{def}\cite{bdre}
Let $X$ be a Banach space and $f : X \rightarrow (-\infty, +\infty]$ be a convex
function. The function $f$ is called sequentially consistent if
for any two sequences $\{x_n\}$ and $\{y_n\}$ in $X$, such that the first one is bounded:
      \begin{equation*}
        \displaystyle\lim_{n\rightarrow\infty}D_f(y_n, x_n)=0\Rightarrow \displaystyle\lim_{n\rightarrow\infty}\|y_n - x_n\|=0.
      \end{equation*}
\end{def}
\begin{lem}\label{2.1}\cite{rsss}
If $f : X \rightarrow \mathbb{R}$ is uniformly Fr\'{e}chet differentiable and
bounded on bounded subsets of $X$, then $\nabla f$ is uniformly continuous on
bounded subsets of $X$ from the strong topology of $X$ to the strong topology
of $X^*$.
\end{lem}
\begin{lem}\label{2.2}\cite{bdian}
If $dom f$ contains at least two points, then the function $f$
is totally convex on bounded sets if and only if the function $f$ is sequentially
consistent.
\end{lem}
\begin{lem}\label{2.10}\cite{bdianz}
Suppose that $f : X \rightarrow (-\infty, +\infty]$ is a Legendre function.
The function $f$ is totally convex on bounded subsets if and only if $f$ is
uniformly convex on bounded subsets.
\end{lem}
\begin{lem}\label{d_g}\cite{aapg}
Let $f:X \rightarrow Y$ be G\^{a}teaux differentiable at any point of $X$. Given $u, v \in X$ such that $[u, v]\subset X$, then
\begin{equation*}
  \|f(u)-f(v)\|\leq\sup\{\|d_Gf(w)\|: w\in [u, v]\}\|u-v\|,
 \end{equation*}
  where $d_Gf(w)$ is called the G\^{a}teaux differential of $f$ at $w$.
\end{lem}

\begin{lem}\label{2.4}\cite{ss}
Let $f : X \rightarrow \mathbb{R}$ be a Legendre function such that $\nabla f^*$ is
bounded on bounded subsets of int $dom f^*$. Let $x_0 \in X$. If $\{D_f (x_0, x_n)\}$ is
bounded, then the sequence $\{x_n\}$ is bounded too.
\end{lem}
\begin{lem}\label{2.5}\cite{rsss}
Suppose that $f$ is G\^{a}teaux differentiable and totally convex on
int $dom f$. Let $x \in$ int $dom f$ and $C \subset$ int $dom f$ be a nonempty, closed and
convex set. If $\hat{x}\in C$, then the following conditions are equivalent:
\begin{enumerate}
  \item [\emph{(i)}] The vector $\hat{x} \in C$ is the Bregman projection of $x$ onto $C$ with respect to
$f$.
  \item [\emph{(ii)}] The vector $\hat{x} \in C$ is the unique solution of the variational inequality:
  $$\langle z-y, \nabla f(z)-\nabla f(x)\rangle \leq 0, \;\;\forall y\in C.$$
  \item [\emph{(iii)}] The vector $\hat{x}$ is the unique solution of the inequality:
  $$D_f (y, z) + D_f (z, x) \leq D_f (y, x),\;\; \forall y\in C.$$
\end{enumerate}
\end{lem}
 The function $V_f : X \times X^* \rightarrow [0,+\infty]$ is defined by
 \begin{equation}\label{vf}
    V_f (x, x^*) = f(x) - \langle x, x^*\rangle + f^*(x^*),\;\;\; \forall x \in X, x^* \in X^*.
 \end{equation}
Therefore
\begin{equation}\label{dfvf}
  V_f (x, x^*) = D_f (x, \nabla f^*(x^*)),\;\;\; \forall x \in X, x^* \in X^*.
\end{equation}
Also, by the subdifferential inequality, we obtain
      \begin{equation}\label{vf<vf}
             V_f (x, x^*) + \langle\nabla f^*(x^*) - x, y^*\rangle\leq V_f (x, x^*+ y^*),
      \end{equation}
for all $x \in X$ and $x^*, y^* \in X^*$ \cite{kftw}.
It is known that $V_f$ is convex in the second variable. Hence, for all $z \in X$, we have
\begin{equation}\label{df<df}
  D_f\bigg{(}z,\nabla f^*\big{(}\sum_{i=1}^{N}t_i\nabla f(x_i)\big{)}\bigg{)}\leq \sum_{i=1}^{N}t_iD_f(z, x_i),
\end{equation}
where $\{x_i\}\subset X$ and $\{t_i\} \subset (0, 1)$ with $\Sigma_{i=1}^{N} t_i = 1$.

The Bregman projection $\overleftarrow{Proj}_{C}^{f}: \text{int}(\text{dom} f )\rightarrow C$ is defined as the necessarily
unique vector $\overleftarrow{Proj}_{C}^{f}(x)\in C$ that satisfies(see \cite{blm})
\begin{equation}\label{inf}
  D_f\big{(}\overleftarrow{Proj}_{C}^{f}(x), x\big{)}= \inf \{D_f(y, x) :y\in C\}.
\end{equation}

Let $X$ be a Banach space. We use $S_X$ to denote the unit sphere $S_X = \{x \in X : \|x\| = 1\}$ and $B_r:=\{y\in X: \|y\|\leq r\}$ for all $r>0$, and $B_r$ is the closed ball in $X$. Then the function $f:X \rightarrow \mathbb{R}$ is said to be uniformly convex on bounded subsets of $X$(see \cite{Zalinescu}) if
$\rho_r(t) > 0$ for all $r, t > 0$, where $\rho_r : [0, \infty) \rightarrow  [0,\infty]$ is defined by
\begin{equation*}
  \rho_r(t)=\displaystyle\inf_{x, y\in B_r,	\|x-y\|=t ,\alpha\in(0,1)} \frac{\alpha f(x) + (1-\alpha)f(y) - f(\alpha x + (1 -\alpha)y)}{\alpha(1-\alpha)},
\end{equation*}
for all $t \geq 0$. The function $\rho_r$ is called the gauge of uniform convexity of $f$.

Now, if $f$ is uniformly convex, then the following lemma is known.
\begin{lem}\label{fsum<sumf}\cite{neyaojc}
Let $X$ be a Banach space, $r > 0$ be a constant and $f : X \rightarrow \mathbb{R}$
be a uniformly convex function on bounded subsets of $X$. Then
\begin{equation*}
  f\bigg{(}\sum_{k=0}^{n}\alpha_kx_k\bigg{)}\leq \sum_{k=0}^{n}\alpha_kf(x_k)-\alpha_i\alpha_j\rho_r(\|x_i - x_j\|),
\end{equation*}
for all $i, j \in \{0, 1, 2, . . . , n\}, x_k \in B_r, \alpha_k \in (0, 1)$ and $k = 0, 1, 2, . . . , n$ with
$\sum_{k=0}^{n}\alpha_k=1$, where $\rho_r$ is the gauge of uniform convexity of $f$.
\end{lem}

A function $f$ on $X$ is said to be coercive\cite{hujb}  if the sublevel set of $f$ is bounded, equivalently, $\lim_{\|x\|\rightarrow\infty}f(x)=\infty$. A function $f$ on $X$ is said to be strongly coercive \cite{Zalinescu} if $\lim_{\|x\|\rightarrow\infty}\frac{f(x)}{\|x\|}=\infty$.
The function $f$ is also said to be uniformly smooth on bounded subsets(\cite{Zalinescu}) if $\lim_{t\rightarrow 0}\frac{\sigma_r(t)}{t}=0$ for all $r > 0$, where $\sigma_r : [0, \infty) \rightarrow [0, \infty]$ is defined
by
\begin{equation*}
  \sigma_r(t)= \displaystyle\sup_{x, y\in B_r,	\|x-y\|=t ,\alpha\in(0,1)} \frac{\alpha f(x + (1-\alpha)ty) + (1-\alpha)f(x-\alpha ty)- f(x)}{\alpha(1-\alpha)},
\end{equation*}
for all $t\geq 0$.
We will need the following Propositions.
\begin{prop}\label{2.8}\cite{Zalinescu}
Let $f : X \rightarrow \mathbb{R}$ be a convex function which is strongly coercive.Then, the following are equivalent:
  \begin{enumerate}
    \item [\emph{(i)}] $f$ is bounded on bounded subsets and uniformly smooth on bounded subsets
of $X$.
    \item [\emph{(ii)}] $f$ is Fr\'{e}chet differentiable and $\nabla f$ is uniformly norm-to-norm continuous
on bounded subsets of $X$.
    \item [\emph{(iii)}] dom $f^* = X^*$, $f^*$ is strongly coercive and uniformly convex on bounded
subsets of $X^*$.
  \end{enumerate}
\end{prop}
\begin{prop}\label{2.9}\cite{Zalinescu}
Let $f : X \rightarrow \mathbb{\mathbb{R}}$ be a convex function which is bounded on
bounded subsets of $X$. Then the following are equivalent:
\begin{enumerate}
  \item [\emph{(i)}] $f$ is strongly coercive and uniformly convex on bounded subsets of $X$.
  \item [\emph{(ii)}] $dom$ $f^* = X^*$, $f^*$ is bounded on bounded subsets and uniformly smooth
on bounded subsets of $X^*$.
  \item [\emph{(iii)}] $dom$ $f^* = X^*$, $f^*$ is Fr\'{e}chet differentiable and $\nabla f^*$ is uniformly norm-to-norm continuous on bounded subsets of $X^*$.
\end{enumerate}
\end{prop}
\begin{lem}\label{2.11}\cite{tjv}
Let $C$ be a nonempty convex subset of X and $f : C \rightarrow \mathbb{R}$ be
a convex and subdifferentiable function on $C$. Then $f$ attains its minimum
at $x\in C$ if and only if $0 \in  \partial f(x) + N_C(x)$, where $N_C(x)$ is the normal cone of $C$ at $x$, that is
\begin{equation*}
  N_C(x):=\{x^*\in X^* : \langle x - y, x^*\rangle\geq 0, \; \forall y\in C\}.
\end{equation*}
\end{lem}
\begin{lem}\label{2.12}\cite{ci}
If f and $g$ are two convex functions on $X$ such that there
is a point $x_0\in$ dom $f$ $\cap$ dom $g$ where $f$ is continuous. Then
\begin{equation*}
  \partial(f + g)(x)= \partial f(x) + \partial g(x), \;\;\; \forall x\in X.
\end{equation*}
\end{lem}

Let $X$ be a real reflexive Banach space and $C$ be a nonempty, closed and convex subset
of $X$. Let $g : C \times C\rightarrow \mathbb{R}$ be a bifunction such that $g(x, x) = 0$ for all $x\in C$. The
equilibrium problem ($EP$) with respect to $g$ on $C$ is stated as follows:
   \begin{equation}\label{EP}
     Find\;  y^*\in C\; \; such\; that \;\; g(x, y^*)\leq 0,\;\; for\; all\; x\in C.
   \end{equation}
The solution set of equilibrium problem \eqref{EP} is denoted by $EP(g)$.

Throughout this paper, we assume that $g : C \times C \rightarrow \mathbb{R}$ is a bifunction satisfying the following
conditions:
\begin{enumerate}
  \item [(A1)] $g$ is monotone, i.e., $g(x, y) +
g(y, x) \leq 0$, for all $x, y \in C$,\\
  \item [(A2)] $g$ is pseudomonotone, i.e., $g(x, y)\geq 0 \Rightarrow g(y, x)\leq0$, for all $x, y \in C$,\\
  \item [(A3)] $g$ is a Bregman-Lipschitz-type continuous function, i.e., there exist two positive
constants $c_1, c_2$, such that
\begin{equation*}
  g(x, y) + g(y, z) \geq g(x, z) - c_1D_f(y, x) - c_2D_f(z, y),\;\; \forall x, y, z \in C,
\end{equation*}
where $f : X \rightarrow (-\infty,+\infty]$ is a Legendre function. The constants $c_1, c_2$ are
called Bregman-Lipschitz coefficients with respect to $f$,
  \item [(A4)] $g$ is weakly  continuous on $C \times C$,\\
  \item [(A5)] $g(x, .)$ is convex, lower semicontinuous and subdifferentiable on $C$ for
every fixed $x \in C$,\\
  \item [(A6)]  $\limsup_{ t\downarrow0} g(tx + (1 - t)y, z) \leq g(y, z)$, for each $x, y, z \in C$.
\end{enumerate}
\begin{re} \cite{yjc} Every monotone bifunction on $C$ is pseudo-monotone but the converse
is not true. A mapping $A : C \rightarrow X^*$ is pseudo-monotone if and only if the
bifunction $g(x, y) = \langle Ax, y - x\rangle$ is pseudo-monotone on $C$.
\end{re}

We denote the set of fixed points of $S$ by $F(S)$, that is $F(S) = \{x \in
C : x \in Sx\}$. A point $x\in C$ is called an asymptotic fixed
point of $S$ if $C$  contains a sequence $\{x_n\}$ which converges weakly to $x$  such that $\lim_{n\rightarrow\infty}\|x_n-Sx_n\|=0$. $\hat{F}(S)$ denotes the set of asymptotic fixed points of $S$ (see \cite{ycsr, srei}).
The mapping $T: C\rightarrow C$ is  called
 Bregman nonexpansive if $D_f (Tx, Ty)\leq D_f(x, y)$, for all $x, y\in C$.
\begin{lem}\cite{egrm}\label{arg}
Let $C$ be a nonempty closed convex subset of a reflexive Banach
space $X$, and $f : X \rightarrow \mathbb{R}$ be a Legendre and strongly coercive function. Suppose
that $g : C \times C \rightarrow \mathbb{R}$ be a bifunction satisfying A2-A5. For the arbitrary sequences $\{x_n\} \subset C$ and $\{\lambda_n\}\subset(0, \infty)$, let $\{y_n\}$ and $\{z_n\}$ be sequences
generated by
\begin{equation}\label{algo3}
    \left\{
    \begin{array}{lr}

       y_n=argmin\{\lambda_n g(x_n, y) +D_f(y, x_n):y\in C\},\\
       z_n= argmin\{\lambda_n g(y_n, y) +D_f(y, x_n):y\in C\}.
       \end{array} \right.
         \end{equation}
          Then, for all $x^*\in EP(g)$\\
          $D_f(x^*, z_n) \leq D_f(x^*, x_n) - (1 - \lambda_nc_1)D_f(y_n, x_n) - (1 - \lambda_nc_2)D_f(z_n, y_n)$.
\end{lem}
\begin{lem}\label{an+1<an}\cite{xuhhk}
Let $\{a_n\}$ be a sequence of nonnegative real numbers satisfying
the inequality:
\begin{equation*}
  a_{n+1}\leq(1-\alpha_n)a_n+\alpha_n\sigma_n, \;\; \forall\geq 0,
\end{equation*}
where
  \begin{enumerate}
    \item [(a)] $\{\alpha_n\}\subset [0, 1]$ and $\sum_{n=0}^{\infty}\alpha_n=\infty$,
    \item [(b)] $\limsup_{n\rightarrow \infty}\sigma_n\leq0, or \sum_{n=0}^{\infty}|\alpha_n\sigma_n|<\infty$.
  \end{enumerate}
Then, $\lim_{n\rightarrow \infty}a_n=0$.
\end{lem}

\begin{lem}\label{2.18}\cite{mpe}
Let $\{a_n\}$ be a sequence of real numbers such that there exists a subsequence $\{n_i\}$ of $\{n\}$ such that $a_{n_i} < a_{n_i+1}$ for all $i\in \mathbb{N}$. Then, there exists a subsequence $\{m_k\} \subset \mathbb{N}$ such that $m_k \rightarrow\infty$ and the following properties are satisfied by all (sufficiently large) numbers $k \in \mathbb{N}$ :
  \begin{equation*}
    a_{m_k} \leq a_{m_k+1}\;\; and \;\; a_k \leq a_{m_k+1}.
  \end{equation*}
In fact, $m_k = \max\{j \leq k : a_j < a_{j+1}\}$.
\end{lem}

We assume that $\varphi : C\times C\rightarrow\mathbb{R}$ is a bifunction satisfying
the following conditions:
\begin{enumerate}
  \item [(B1)] $\varphi(x, x)=0$, for all $x\in C$.
  \item [(B2)] $\varphi$ is monotone, i.e. $\varphi(x, y) + \varphi(y, x)\leq 0$, for all $x, y\in C$.
  \item [(B3)] $\displaystyle\lim_{t\downarrow 0}\varphi(tz + (1-t)x, y)\leq \varphi(x, y)$, for all $x, y, z\in C$.
  \item [(B4)] for all $x\in C$, $y\mapsto \varphi(x, y)$ is convex and lower semicontinuous.
\end{enumerate}
The resolvent of $\varphi$ is the operator $Res^f_\varphi : X\rightarrow 2^C$ defined by \cite{Combettes, rssst}
\begin{equation}\label{res2}
  Res^f_\varphi x=\{z\in C: \varphi(z, y) + \langle \nabla f(z) - \nabla f(x), y-z\rangle\geq 0, \;\;\forall y\in C\}.
\end{equation}
If $f: X\rightarrow (-\infty, \infty]$ be a G\^{a}teaux differentiable and strongly
coercive function and $\varphi$ satisfies conditions B1 - B4. Then $dom$ $Res^f_\varphi = X$ (see \cite[Lemma 1]{rssst}).
\begin{lem}\label{res3}
\cite{rssst}
Suppose that $X$ is a real reflexive Banach space and $C$ a nonempty closed convex subset of $X$. Let $f : X\rightarrow (-\infty, \infty]$ be a Legendre function. If $\varphi$ be a bifunction from  $C\times C$ to $\mathbb{R}$ satisfies B1 - B4, then the followings
hold:
\begin{enumerate}
  \item [(i)] $Res^f_\varphi$ is single-valued.
  \item [(ii)] $Res^f_\varphi$ is a Bregman firmly nonexpansive operator.
  \item [(iii)] $F(Res^f_\varphi)= EP(\varphi)$.
  \item [(iv)]  $EP(\varphi)$ is a closed and convex subset of $C$.
  \item [(v)]  $ D_f(u, x)\geq D_f(u, Res^f_\varphi(x)) + D_f(Res^f_\varphi(x), x)$, for all $x\in X$ and $u\in F(Res^f_\varphi)$.
\end{enumerate}
\end{lem}
  \section{Main results}
 \begin{thm}\label{maina}
Let $C$ be a nonempty closed convex subset of a real reflexive Banach
space $X$, and let $f : X \rightarrow \mathbb{R}$ be an admissible, strongly coercive Legendre function which
is bounded, uniformly Fr\'{e}chet differentiable and totally convex on bounded
subsets of $X$. Let  $g : C\times C \rightarrow \mathbb{R}$ be a bifunction satisfying
A1-A5. Assume that $S : C \rightarrow C$ is a Bregman nonexpansive mapping with $\hat{F}(S)=F(S)$. Let $\varphi : C\times C\rightarrow\mathbb{R}$ be a bifunction satisfying
 B1-B4 and $\Omega =F(S)\cap EP(g) \cap F(Res^f_\varphi)\neq\emptyset$. Suppose that $\{x_n\}$ is a sequence generated by $x_1\in C$, $u\in X$ and

    \begin{equation}\label{algo3}
    \left\{
    \begin{array}{lr}

       y_n=argmin\{\lambda_n g(x_n, y) +D_f(y, x_n):y\in C\},\\
       z_n= argmin\{\lambda_n g(y_n, y) +D_f(y, x_n):y\in C\},\\
      % t_n=argmin\{\lambda_n g(z_n, y) +D_f(y, x_n):y\in C\},\\
       %w_n=\nabla f^*(\gamma_{n,1}\nabla f(x_n)+\gamma_{n,2}\nabla f(z_n) +\gamma_{n,3}\nabla f(Sz_n)),\\
       v_n=\nabla f^*(\delta_n\nabla f(Res^f_\varphi x_n)+(1-\delta_n)\nabla f(Res^f_\varphi z_n)),\\
       w_n=\nabla f^*(\gamma_{n,1}\nabla f(v_n)+\gamma_{n,2}\nabla f(z_n) +\gamma_{n,3}\nabla f(Sz_n)),\\
       u_n\in C \;\;such\;that\\
       \varphi(u_n, y) + \langle \nabla f(u_n)-\nabla f(w_n), y-u_n\rangle\geq 0,\;\;\; \forall\;\; y\in C,\\
   %    C_{n+1}=\{v\in C_n: D_f(v, u_n)\leq D_f(v, x_n)\},\\
 %      D_n=\{z\in C: \langle x_n-z, x_0 -x_n\rangle \geq 0\},\\
        k_n=\nabla f^*(\beta_n\nabla f(w_n)+(1-\beta_n)\nabla f(Sw_n)),\\
        h_n=\nabla f^* (\alpha_{n,1}\nabla f( u)+\alpha_{n,2}\nabla f(x_n)+\alpha_{n,3}\nabla f(u_n)+\alpha_{n,4}\nabla f(k_n)),\\
       x_{n+1}=\overleftarrow{Proj}_C^f h_n.
    \end{array} \right.
   \end{equation}
   where $\{\delta_n\}\subset(0, 1)$ and $\{\alpha_{n,i}\}^4_{i=1}$,  $\{\beta_n\}$ and $\{\lambda_n\}$ satisfy the following conditions:
     \begin{enumerate}
       \item [\emph{(i)}] $\{\alpha_{n,i}\}\subset (0, 1), \sum_{i=1}^{4}\alpha_{n,i}=1,  \lim_{n\rightarrow \infty}\alpha_{n,1}=0,$ $ \Sigma_{n=1}^{\infty}\alpha_{n,1}=\infty$, $\liminf_{n\rightarrow\infty} \alpha_{n,i}\alpha_{n,j}>0$ for all $i\neq j$ and $2\leq i,j\leq3$.
       \item [\emph{(ii)}] $\{\gamma_{n,i}\}\subset(0, 1)$, $\gamma_{n,1}+\gamma_{n,2}+\gamma_{n,3}=1$, $\liminf_{n\rightarrow\infty} \gamma_{n,i}\gamma_{n,j}>0$ for all $i\neq j$ and $1\leq i,j\leq3$.
       \item [\emph{(iii)}] $\{\lambda_n\}\subset [a, b]\subset (0, p)$, where $p=\min\{\frac{1}{c_1}, \frac{1}{c_2}\}$, $c_1, c_2$ are the Bregman-Lipschitz coefficients of $g$.
       \item [\emph{(iv)}] $\{\beta_n\}\subset (0, 1), \liminf_{n\rightarrow \infty}\beta_n(1-\beta_n)>0$.
     \end{enumerate}
     Then, the sequence $\{x_n\}$ converges strongly to $\overleftarrow{Proj}_{\Omega}^fu$.
\end{thm}
   \begin{proof}
   First, we show that $\Omega$ is a closed and convex subset of $C$.
 It follows from  \cite[Lemma $2.14$]{egrm}  that  $EP(g)$ is a convex and weakly closed (so closed) subset of $C$. By the conditions (iii) and (iv) of  Lemma \ref{res3}, $F(Res^f_\varphi)$ is a closed and convex subset of $C$. Also, similar to \cite[Proposition $3.1$]{snzh} $F(S)$ is a closed and convex subset of $C$.
    Hence, $\Omega$ is closed and convex subset of $C$.

 Let $\hat{u}=\overleftarrow{Proj}_{\Omega}^fu$. By \eqref{df<df}, Lemma \ref{arg} (v) of Lemma \ref{res3}, we have
     \begin{align}\label{res1}
       D_f(\hat{u}, v_n)  =& D_f(\hat{u}, \nabla f^*(\delta_n\nabla f(Res^f_\varphi x_n)+(1-\delta_n)\nabla f(Res^f_\varphi z_n))\nonumber\\
       \leq&  \delta_nD_f(\hat{u}, Res^f_\varphi x_n)  + (1-\delta_n)D_f(\hat{u}, Res^f_\varphi z_n)\nonumber\\
       \leq&  \delta_nD_f(\hat{u}, x_n)  + (1-\delta_n)D_f(\hat{u}, z_n)\nonumber\\
       \leq& D_f(\hat{u}, x_n).
     \end{align}
Then from \eqref{df<df}, \eqref{res1}, Lemma \ref{arg} and the Bregman nonexpansiveness of $S$, we conclude that
    \begin{align}\label{wnxn0}
       D_f(\hat{u}, w_n)  =& D_f(\hat{u}, \nabla f^*(\gamma_{n,1}\nabla f(v_n)+\gamma_{n,2}\nabla f(z_n) +\gamma_{n,3}\nabla f(Sz_n)))\nonumber \\
       \leq &\gamma_{n,1}D_f( \hat{u}, v_n)+\gamma_{n,2}D_f( \hat{u}, z_n)+\gamma_{n,3}D_f(\hat{u}, Sz_n)\nonumber\\
       \leq&\gamma_{n,1}D_f( \hat{u}, v_n)+\gamma_{n,2}D_f( \hat{u}, z_n)+\gamma_{n,3}D_f(\hat{u}, z_n)\nonumber\\
       \leq&D_f( \hat{u}, x_n).
    \end{align}
    It follows from  \eqref{df<df} and the Bregman nonexpansiveness of $S$ that
    \begin{align}\label{knxn0}
       D_f(\hat{u}, k_n)  =& D_f(\hat{u}, \nabla f^*(\beta_n\nabla f(w_n)+(1-\beta_n)\nabla f(Sw_n) ))\nonumber \\
       \leq &\beta_nD_f( \hat{u}, w_n)+(1-\beta_n)D_f( \hat{u}, Sw_n)\nonumber\\
       \leq&\beta_nD_f( \hat{u}, w_n)+(1-\beta_n)D_f( \hat{u}, w_n)\nonumber\\
       \leq&D_f( \hat{u}, w_n).
    \end{align}
  Therefore, from  \eqref{wnxn0},
  \begin{equation}\label{kn<xn}
    D_f(\hat{u}, k_n)\leq D_f(\hat{u}, x_n).
  \end{equation}

By \eqref{res2} and the algorithm \eqref{algo3}, we obtain that $u_n\in Res^f_\varphi w_n$. It follows from Lemma \ref{res3} that $Res^f_\varphi$ is single valued, hence $u_n=Res^f_\varphi w_n$.  So from part (v) of Lemma \ref{res3}, we have
\begin{equation}\label{res4}
  D_f(\hat{u}, u_n)=D_f(\hat{u}, Res^f_\varphi w_n)\leq D_f(\hat{u}, w_n).
\end{equation}
 Then, from \eqref{df<df}, \eqref{knxn0}, \eqref{res4} and part (iii) of Lemma \ref{2.5}, we obtain
         \begin{align}\label{decrsing}
           D_f(\hat{u}, x_{n+1})  \leq& D_f(\hat{u}, h_n)\nonumber\\
            =&D_f(\hat{u}, \nabla f^* (\alpha_{n,1}\nabla f( u)
           + \alpha_{n,2}\nabla f(x_n)+\alpha_{n,3}\nabla f(u_n)+\alpha_{n,4}\nabla f(k_n)))\nonumber \\
           \leq & \alpha_{n,1} D_f(\hat{u}, u)+ \alpha_{n,2}D_f(\hat{u},x_n)+(\alpha_{n,3} +\alpha_{n,4})D_f(\hat{u}, w_n).
          \end{align}
Therefore, it follows from \eqref{wnxn0} that
           \begin{align*}\label{eq3}
              D_f(\hat{u}, x_{n+1})\leq& \alpha_{n,1} D_f(\hat{u}, u)+ (1-\alpha_{n,1})D_f(\hat{u},x_n)\\
              \leq&\max\{D_f(\hat{u}, u), D_f(\hat{u}, x_n)\},
             \end{align*}
hence, by the induction process, we conclude that
\begin{equation*}
   D_f(\hat{u}, x_{n+1})\leq \max\{D_f(\hat{u}, u), D_f(\hat{u}, x_1)\}.
\end{equation*}
Now from the above, we have that the sequence $\{D_f(\hat{u}, x_n)\}$ is bounded. From Lemma \ref{2.10}, $f$ is a uniformly convex function on bounded subsets. Hence, the condition (i) of Proposition \ref{2.9} holds, equivalently, the condition (ii) of the proposition holds, i.e., $f^*$ is bounded on bounded subsets of $X^*$. Thus, $\nabla f^*$ is also bounded on bounded subsets of $X^*$ (see \cite[Proposition 7.8]{bhhbjm}). From Lemma \ref{2.4}, we conclude that the sequence $\{x_n\}$ is bounded. It follows from \eqref{res1}, \eqref{wnxn0}, \eqref{res4}, Lemma \ref{arg} and boundednes of $\{D_f(\hat{u}, x_n)\}$ that $\{D_f(\hat{u}, z_n)\}$, $\{D_f(\hat{u}, v_n)\}$, $\{D_f(\hat{u}, w_n)\}$ and $\{D_f(\hat{u}, u_n)\}$ are bounded, hence from Lemma \ref{2.4}, we conclude that the sequences $\{z_n\}$, $\{v_n\}$, $\{w_n\}$ and $\{u_n\}$ are bounded.
Hence  $\{\nabla f(z_n)\}$, $\{\nabla f(v_n)\}$ and $\{\nabla f(w_n)\}$ are bounded (see \cite[Proposition 1.1.11]{bdian}). It follows from algorithm \eqref{algo3} that
$\gamma_{n,3}\nabla f(Sz_n)=\nabla f(w_n)-\gamma_{n,1}\nabla f(v_n)-\gamma_{n,2}\nabla f(z_n)$. Then from (ii), $\{\nabla f(Sz_n)\}$ is bounded.

Since $f$ is uniformly Fr\'{e}chet differentiable and bounded on bounded subsets of $X$, then from Lemma \ref{2.1}, $\nabla f$ is uniformly norm-to-norm continuous
on bounded subsets of $X$. Also from the assumption that $f$ is a convex and strongly coercive function, we may conclude that the condition (ii) of Proposition \ref{2.8} holds, equivalently, the condition (iii) of the proposition holds, i.e., $f^*$ is uniformly convex on bounded subsets of $X^*$. Let $r_1= \displaystyle\sup_n\{\|\nabla f(z_n)\|, \|\nabla f(Sz_n)\|, \|\nabla f(v_n)\|\}$. Since $\{\nabla f(z_n)\}$, $\{\nabla f(Sz_n)\}$ and $\{\nabla f(v_n)\}$ are  bounded sequences then $r_1<\infty$.
 Hence, by \eqref{vf}, \eqref{dfvf}, \eqref{res1}, Lemmas \ref{fsum<sumf}, \ref{arg} and the Bregman nonexpansiveness of $S$, we have that
     \begin{align}\label{wnxn}
       D_f(\hat{u}, w_n)  =& V_f(\hat{u}, \gamma_{n,1}\nabla f(v_n)+\gamma_{n,2}\nabla f(z_n) +\gamma_{n,3}\nabla f(Sz_n))\nonumber \\
       =&  f(\hat{u}) - \langle\hat{u}, \gamma_{n,1}\nabla f(v_n)+\gamma_{n,2}\nabla f(z_n) +\gamma_{n,3}\nabla f(Sz_n) \rangle \nonumber\\
       &+ f^*(\gamma_{n,1}\nabla f(v_n)+\gamma_{n,2}\nabla f(z_n) +\gamma_{n,3}\nabla f(Sz_n))\nonumber \\
       \leq &f(\hat{u})-\gamma_{n,1}\langle\hat{u}, \nabla f(v_n) \rangle- \gamma_{n,2}\langle \hat{u}, \nabla f(z_n)\rangle -\gamma_{n,3}\langle \hat{u}, \nabla f(Sz_n)\rangle\nonumber \\
       &+\gamma_{n,1}f^*(\nabla f(v_n))+\gamma_{n,2}f^*(\nabla f(z_n))+\gamma_{n,3}f^*(\nabla f(Sz_n))\nonumber\\
       &- \gamma_{n,2}\gamma_{n,3}\rho_{r_1}^{*}(\|\nabla f(z_n) -\nabla f(Sz_n)\|)\nonumber\\
       = &\gamma_{n,1}V_f( \hat{u}, \nabla f(v_n))+ \gamma_{n,2}V_f( \hat{u}, \nabla f(z_n))+\gamma_{n,3}V_f(\hat{u}, \nabla f(Sz_n))\nonumber\\
       &- \gamma_{n,2}\gamma_{n,3}\rho_{r_1}^{*}(\|\nabla f(z_n) -\nabla f(Sz_n)\|)\nonumber\\
       =&\gamma_{n,1}D_f( \hat{u}, v_n)+\gamma_{n,2}D_f( \hat{u}, z_n)+\gamma_{n,3}D_f(\hat{u}, Sz_n)\nonumber\\
       &- \gamma_{n,2}\gamma_{n,3}\rho_{r_1}^{*}(\|\nabla f(z_n) -\nabla f(Sz_n)\|)\nonumber\\
       \leq&\gamma_{n,1}D_f( \hat{u}, v_n)+\gamma_{n,2}D_f( \hat{u}, z_n)+\gamma_{n,3}D_f(\hat{u}, z_n)\nonumber\\
       &- \gamma_{n,2}\gamma_{n,3}\rho_{r_1}^{*}(\|\nabla f(z_n) -\nabla f(Sz_n)\|)\nonumber\\
        \leq& D_f(\hat{u}, x_n)- \gamma_{n,2}\gamma_{n,3}\rho_{r_1}^{*}(\|\nabla f(z_n) -\nabla f(Sz_n)\|),
         \end{align}
where $\rho^*_{r_1}$ is the gauge of uniform convexity of $f^*$. Also, by repeating the above process, we obtain
\begin{equation}\label{vnzn}
  D_f(\hat{u}, w_n) \leq D_f(\hat{u}, x_n)- \gamma_{n,1}\gamma_{n,2}\rho_{r_1}^{*}(\|\nabla f(v_n) -\nabla f(z_n)\|).
\end{equation}

It follows from \eqref{knxn0} that $\{D_f(\hat{u}, k_n)\}$ is bounded, hence from Lemma \ref{2.4}, we conclude that $\{k_n\}$ is bounded. Therefore, $\{\nabla f(k_n)\}$ is bounded. Hence from $(1-\beta_n)\nabla f(Sw_n)=\nabla f(k_n)-\beta_n\nabla f(w_n)$, we obtain that $\{\nabla f(Sw_n)\}$ is bounded.
In a similar way as above, from \eqref{vf}, \eqref{dfvf}, \eqref{wnxn0} and Lemma \ref{fsum<sumf}, there exists a number $r_2= \displaystyle\sup_n\{\|\nabla f(w_n)\|, \|\nabla f(Sw_n)\|\}$ such that
\begin{equation}\label{knxn}
  D_f(\hat{u}, k_n)\leq D_f(\hat{u}, x_n)- \beta_n(1-\beta_n)\rho_{r_2}^{*}(\|\nabla f(w_n) -\nabla f(Sw_n)\|),
\end{equation}
        Now, from \eqref{decrsing} and \eqref{wnxn}, we have
          \begin{align}\label{eq2}
             D_f(\hat{u}, x_{n+1})\leq &\alpha_{n,1} D_f(\hat{u}, u)+\alpha_{n,2}D_f(\hat{u}, x_n)+(\alpha_{n,3} + \alpha_{n,4})D_f(\hat{u}, x_n)\nonumber\\
             &-(\alpha_{n,3} + \alpha_{n,4})\gamma_{n,2}\gamma_{n,3}\rho_{r_1}^{*}(\|\nabla f(z_n) -\nabla f(Sz_n)\|)\nonumber\\
             =&\alpha_{n,1} D_f(\hat{u}, u)+(1-\alpha_{n,1})D_f(\hat{u}, x_n)\nonumber\\
             &-(\alpha_{n,3} + \alpha_{n,4})\gamma_{n,2}\gamma_{n,3}\rho_{r_1}^{*}(\|\nabla f(z_n) -\nabla f(Sz_n)\|).
          \end{align}
Also, it is implied  from \eqref{decrsing} and \eqref{vnzn} that
\begin{align}\label{vnzn1}
  D_f(\hat{u}, x_{n+1})\leq&\alpha_{n,1} D_f(\hat{u}, u)+(1-\alpha_{n,1})D_f(\hat{u}, x_n)\nonumber\\
             &-(\alpha_{n,3} + \alpha_{n,4})\gamma_{n,1}\gamma_{n,2}\rho_{r_1}^{*}(\|\nabla f(v_n) -\nabla f(z_n)\|).
\end{align}
Moreover, from \eqref{df<df}, \eqref{wnxn0}, \eqref{res4}, part (iii) of Lemma \ref{2.5} and \eqref{knxn}, we obtain
          \begin{align}\label{eq3}
              D_f(\hat{u}, x_{n+1})\leq&D_f(\hat{u}, h_n)\nonumber\\
              \leq&\alpha_{n,1}D_f(\hat{u}, u)+\alpha_{n,2}D_f(\hat{u}, x_n)+\alpha_{n,3}D_f(\hat{u}, u_n)+\alpha_{n,4}D_f(\hat{u}, k_n)\nonumber\\
              \leq&\alpha_{n,1}D_f(\hat{u}, u)+(1-\alpha_{n,1})D_f(\hat{u}, x_n)\nonumber\\
              &-\alpha_{n,4}\beta_n(1-\beta_n)\rho_{r_2}^{*}(\|\nabla f(w_n) -\nabla f(Sw_n)\|).
          \end{align}
From \eqref{df<df}, \eqref{res1}, Lemma \ref{arg} and the Bregman nonexpansiveness of $S$, we have
      \begin{align}\label{wnxn1}
       D_f(\hat{u}, w_n)
       \leq&  \gamma_{n,1}D_f( \hat{u}, v_n)+\gamma_{n,2}D_f( \hat{u}, z_n)+\gamma_{n,3}D_f(\hat{u}, Sz_n)\nonumber\\
       \leq& \gamma_{n,1}D_f( \hat{u}, x_n)+(\gamma_{n,2}+\gamma_{n,3})D_f(\hat{u}, z_n)\nonumber\\
       \leq&  \gamma_{n,1}D_f( \hat{u}, x_n)+(\gamma_{n,2}+\gamma_{n,3})\big{(}D_f(\hat{u},x_n)\nonumber\\
       &-(1-\lambda_nc_1)D_f(y_n, x_n)-(1-\lambda_nc_2)D_f(z_n, y_n)\big{)}\nonumber\\
       =& D_f( \hat{u}, x_n)-(\gamma_{n,2}+\gamma_{n,3})\big{(}(1-\lambda_nc_1)D_f(y_n, x_n)\nonumber\\
       &+(1-\lambda_nc_2)D_f(z_n, y_n)\big{)}.
         \end{align}
         From \eqref{decrsing} and \eqref{wnxn1}, we have
         \begin{align*}%\label{o}
            D_f(\hat{u}, x_{n+1}) \leq & \alpha_{n,1} D_f(\hat{u}, u)+\alpha_{n,2}D_f(\hat{u}, x_n)+ (\alpha_{n,3}+\alpha_{n,4})D_f(\hat{u}, w_n)\nonumber\\
           \leq & \alpha_{n,1} D_f(\hat{u}, u)+\alpha_{n,2}D_f(\hat{u}, x_n)+(\alpha_{n,3}+\alpha_{n,4})\big{[}D_f( \hat{u}, x_n)\nonumber\\
            &-(\gamma_{n,2} +\gamma_{n,3})\big{(}(1-\lambda_nc_1)D_f(y_n, x_n)+
              (1-\lambda_nc_2)D_f(z_n, y_n)\big{)}\big{]}\nonumber\\
             =&\alpha_{n,1} D_f(\hat{u}, u)+(1-\alpha_{n,1})D_f(\hat{u}, x_n)-(\alpha_{n,3}+\alpha_{n,4})
            (\gamma_{n,2}\nonumber\\ &+\gamma_{n,3})\big{(}(1-\lambda_nc_1)D_f(y_n, x_n)+
              (1-\lambda_nc_2)D_f(z_n, y_n)\big{)}.
         \end{align*}
       Thus it follows that
        \begin{align}\label{l}
          (\alpha_{n,3}+\alpha_{n,4})(\gamma_{n,2}+\gamma_{n,3})(1-&\lambda_nc_1)D_f(y_n, x_n) \nonumber\\
           \leq& \alpha_{n,1} D_f(\hat{u}, u)+D_f( \hat{u}, x_n)-D_f(\hat{u}, x_{n+1})\nonumber\\
          &-(\alpha_{n,3}+\alpha_{n,4})(\gamma_{n,2}+\gamma_{n,3})(1-\lambda_nc_2)D_f(z_n, y_n).
        \end{align}

Now by \eqref{dfvf}, \eqref{vf<vf}, \eqref{df<df}, \eqref{wnxn0}, \eqref{kn<xn},  \eqref{res4} and the part (iii) of Lemma \ref{2.5}, we conclude that
         \begin{align}\label{u,xn+1}
           D_f(\hat{u}, x_{n+1})  \leq& D_f(\hat{u}, \nabla f^* (\alpha_{n,1}\nabla f( u)+\alpha_{n,2}\nabla f(x_n)+\alpha_{n,3}\nabla f(u_n)+\alpha_{n,4}\nabla f(k_n))\nonumber\\
           =&V_f(\hat{u}, \alpha_{n,1}\nabla f( u)+\alpha_{n,2}\nabla f(x_n)+\alpha_{n,3}\nabla f(u_n)+\alpha_{n,4}\nabla f(k_n)) \nonumber\\
           \leq & V_f(\hat{u}, \alpha_{n,1}\nabla f( u)+\alpha_{n,2}\nabla f(x_n)+\alpha_{n,3}\nabla f(u_n)+\alpha_{n,4}\nabla f(k_n)\nonumber\\
           &-\alpha_{n,1}(\nabla f(u)- \nabla f(\hat{u}))) +\alpha_{n,1}\langle h_n -\hat{u}, \nabla f( u)- \nabla f(\hat{u})\rangle\nonumber\\
           =& V_f(\hat{u}, \alpha_{n,2}\nabla f(x_n)+\alpha_{n,3}\nabla f(u_n)+ \alpha_{n,4}\nabla f(k_n)+\alpha_{n,1}\nabla f(\hat{u}))\nonumber\\
            & +\alpha_{n,1}\langle h_n -\hat{u}, \nabla f( u)- \nabla f(\hat{u})\rangle\nonumber\\
            =& D_f(\hat{u}, \nabla f^*(\alpha_{n,1}\nabla f( \hat{u})+\alpha_{n,2}\nabla f(x_n)+\alpha_{n,3}\nabla f(u_n)+ \alpha_{n,4}\nabla f(k_n)))\nonumber\\
            &+\alpha_{n,1}\langle h_n -\hat{u}, \nabla f( u)- \nabla f(\hat{u})\rangle\nonumber\\
            \leq& \alpha_{n,1}D_f(\hat{u}, \hat{u})+\alpha_{n,2}D_f(\hat{u}, x_n)+ \alpha_{n,3}D_f(\hat{u}, u_n)+ \alpha_{n,4}D_f(\hat{u}, k_n)\nonumber\\
            &+\alpha_{n,1}\langle h_n -\hat{u}, \nabla f( u)- \nabla f(\hat{u})\rangle\nonumber\\
            \leq& \alpha_{n,2}D_f(\hat{u}, x_n)+ \alpha_{n,3}D_f(\hat{u}, x_n)+ \alpha_{n,4}D_f(\hat{u}, x_n)\nonumber\\
            &+\alpha_{n,1}\langle h_n -\hat{u}, \nabla f( u)- \nabla f(\hat{u})\rangle\nonumber\\
             =&(1-\alpha_{n,1})D_f(\hat{u}, x_n)+\alpha_{n,1}\langle h_n -\hat{u}, \nabla f( u)- \nabla f(\hat{u})\rangle.
         \end{align}
Also, from \eqref{df<df}, we have
         \begin{align}\label{m}
           D_f(w_n, h_n)\leq & \alpha_{n,1} D_f(w_n, u)+\alpha_{n,2}D_f(w_n, x_n)+\alpha_{n,3} D_f(w_n, u_n)+\alpha_{n,4}D_f(w_n, k_n)\nonumber\\
           \leq & \alpha_{n,1} D_f(w_n, u)+\alpha_{n,2}D_f(w_n, x_n)+\alpha_{n,3} D_f(w_n, u_n)\nonumber\\
           &+ \alpha_{n,4}\big{(}\beta_nD_f(w_n, w_n)+ (1-\beta_n)D_f(w_n, Sw_n)\big{)}\nonumber\\
            \leq& \alpha_{n,1} D_f(w_n, u)+\alpha_{n,2}D_f(w_n, x_n)+\alpha_{n,3} D_f(w_n, u_n)\nonumber\\
            &+ \alpha_{n,4}(1-\beta_n)D_f(w_n, Sw_n).
         \end{align}
We now consider the following two possible cases.

\section*{Case1}
There exists some $n_0\in \mathbb{N}$ such that $\{D_f (\hat{u}, x_n)\}$ is nonincreasing for all $n \geq n_0$. Therefore, $\displaystyle\lim_{n\rightarrow \infty}D_f (\hat{u}, x_n)$ exists and $D_f (\hat{u}, x_n) - D_f (\hat{u}, x_{n+1})\rightarrow 0$ as $n\rightarrow \infty$.
From \eqref{l}, the conditions (i), (ii) and (iii), it follows that $D_f(y_n, x_n)\rightarrow 0$ as $n\rightarrow \infty$, then by Lemma \ref{2.2} and boundedness of $\{x_n\}$, we obtain
      \begin{equation}\label{ynxn}
        \displaystyle\lim_{n\rightarrow\infty}\|y_n - x_n\|=0.
      \end{equation}
In a similar way, from \eqref{l}, we have
      \begin{equation}\label{ynzntn}
        \displaystyle\lim_{n\rightarrow\infty}\|z_n - y_n\|=0.
      \end{equation}
      By \eqref{ynxn} and \eqref{ynzntn}, we have
      \begin{equation}\label{total}
         \displaystyle\lim_{n\rightarrow\infty}\|z_n - x_n\|=0.
      \end{equation}
     It follows from \eqref{ynxn} that the sequence $\{y_n\}$ is bounded. Next, from \eqref{eq2} and the condition (i), we have
     \begin{align}\label{tnstn}
       (\alpha_{n,3}+\alpha_{n,4})\gamma_{n,2}\gamma_{n,3}&\rho_{r_1}^{*}(\|\nabla f(z_n) -\nabla f(Sz_n)\|)\nonumber\\
         \leq&\alpha_{n,1}D_f(\hat{u}, u)+(1-\alpha_{n,1})D_f(\hat{u}, x_n)- D_f(\hat{u}, x_{n+1})\nonumber\\
         \leq &\alpha_{n,1}D_f(\hat{u}, u)+D_f(\hat{u}, x_n)- D_f(\hat{u}, x_{n+1})\rightarrow 0, \;\; as\;\; n\rightarrow\infty.
     \end{align}
Hence, it follows from the conditions (i) and (ii) that $\displaystyle\lim_{n\rightarrow\infty}\rho_{r_1}^{*}(\|\nabla f(z_n) -\nabla f(Sz_n)\|)=0$.
Furthermore, we claim that $\displaystyle\lim_{n\rightarrow\infty}\|\nabla f(z_n) -\nabla f(Sz_n)\|=0$. If not, there exist
a subsequence $\{n_m\}$ of $\{n\}$ and a positive number $\epsilon_1$ such that  $\|\nabla f(z_{n_m}) -\nabla f(Sz_{n_m})\|\geq\epsilon_1$.
 Therefore $\rho_{r_1}^{*}(\|\nabla f(z_{n_m}) -\nabla f(Sz_{n_m})\|)\geq\rho_{r_1}^{*}(\epsilon_1)$ for all $m \in \mathbb{N}$, because $\rho_{r_1}^{*}$ is nondecreasing. Letting $m\rightarrow\infty$, we conclude that $\rho_{r_1}^{*}(\epsilon_1)\leq0$ which is a contradiction to the uniform convexity of $f^*$ on bounded subsets of $X^*$. Therefore $\displaystyle\lim_{n\rightarrow\infty}\|\nabla f(z_n) -\nabla f(Sz_n)\|=0$. By Proposition \ref{2.9} and Lemma \ref{2.10}, $\nabla f^*$ is uniformly norm-to-norm continuous on bounded subsets of $X^*$. Then
\begin{equation}\label{tnstn1}
  \displaystyle\lim_{n\rightarrow\infty}\|z_n-Sz_n\|=0.
\end{equation}
From \eqref{vnzn1} and \eqref{eq3}, the conditions (ii), (iv) and a similar technique as in the above, we have
\begin{equation}\label{wntwn}
  \displaystyle\lim_{n\rightarrow\infty}\|v_n-z_n\|=0\;\;\&\;\;\displaystyle\lim_{n\rightarrow\infty}\|w_n-Sw_n\|=0.
\end{equation}

Now, since $\{x_n\}$ is bounded and $X$ is reflexive, then by \cite[Theorem 1.2.14]{Takahashi}, there exists a
subsequence $\{x_{n_m}\}$ of $\{x_n\}$ and a point $q\in X$ such that $x_{n_m}\rightharpoonup q$. Therefore by \eqref{total}, $z_{n_m}\rightharpoonup q$. Using \eqref{tnstn1}, we obtain $q\in \hat{F}(S)=F(S)$.

Next, we show that $q\in EP(g)$.
For this purpose, let $x, y\in \text{int}$ $\text{dom} f$. We prove that $\nabla f(x)-\nabla f(y)\in \partial_1D_f(x, y)$ which $\partial_1 D_f(x, y)$ is the subdifferential of $D_f(x, y)$ at the first component.
We let $\vartheta:=\nabla f(x)-\nabla f(y)$. Note that
\begin{align*}
   0\geq -D_f(z, x)=& f(x)-f(z)-\langle x-z, \nabla f(x)\rangle\\
       =& f(x)-f(z)-\langle x-z, \nabla f(y)\rangle -\langle x-z, \vartheta\rangle\\
    =& f(x)-f(z)-\langle x-y, \nabla f(y)\rangle-\langle y-z, \nabla f(y)\rangle -\langle x-z, \vartheta\rangle\\
    = &f(x)-f(y)-\langle x-y, \nabla f(y)\rangle+ f(y) - f(z)\\
    &-\langle y-z, \nabla f(y)\rangle -\langle x-z, \vartheta\rangle\\
    = &D_f(x, y) - D_f(z, y)-\langle x-z, \vartheta\rangle,
\end{align*} for all $z\in \text{dom} f$. Then
\begin{equation*}\label{rond}
  D_f(x, y)+\langle z-x, \vartheta\rangle\leq D_f(z, y),
\end{equation*}
 for all $z\in \text{dom} f$.
Therefore,
 $\nabla f(x)-\nabla f(y)=\vartheta\in\partial_1 D_f(x, y)$. On the other hand, using the fact that $D_f (., x)$ is convex, differentiable and lower semicontinuous on int dom$f$, it follows from Proposition \ref{singdf} that $\partial_1D_f(x, y)=\{\nabla f(x)-\nabla f(y)\}$.
Since $y_n=argmin\{\lambda_n g(x_n, y) +D_f(y, x_n):y\in C\}$, from the condition $A5$, Lemmas \ref{2.11} and \ref{2.12}, it follows that
\begin{align*}
  0 \in \partial\{&\lambda_n g(x_n, y_n) +D_f(y_n, x_n)\}+ N_C(y_n)\\
  =&\lambda_n\partial_2 g(x_n, y_n) +\partial_1 D_f(y_n, x_n)+ N_C(y_n).
\end{align*}
Hence, there exist $\zeta_n\in\partial_2 g(x_n, y_n)$ and $\eta_n\in N_C(y_n)$ such that
\begin{equation}\label{zta}
  \lambda_n \zeta_n +\nabla f(y_n) - \nabla f(x_n) + \eta_n =0.
\end{equation}
Since $\eta_n \in N_C(y_n)$, then
\begin{equation}\label{nc}
  \langle y_n - z, \eta_n \rangle\geq 0, \;\; \forall z\in C,
\end{equation}
from \eqref{zta} and \eqref{nc}, we have
   \begin{align*}
     0\leq\langle y_n - z,&  -\lambda_n \zeta_n -\nabla f(y_n) + \nabla f(x_n) \rangle\  \\
     =& \lambda_n\langle z - y_n,  \zeta_n\rangle -\langle y_n - z, \nabla f(y_n) - \nabla f(x_n) \rangle,
   \end{align*}
for all $z\in C$. Therefore
\begin{equation}\label{ep}
  \langle y_n - z, \nabla f(y_n) - \nabla f(x_n) \rangle\leq \lambda_n\langle z - y_n,  \zeta_n\rangle.
\end{equation}
Moreover, since $\zeta_n \in \partial_2 g(x_n, y_n)$, then
\begin{equation}\label{ep1}
  g(x_n, y_n)-g(x_n, z)\leq\langle y_n - z, \zeta_n\rangle,
\end{equation}
for all $z \in C$. Now, from \eqref{ep} and \eqref{ep1}, we have
\begin{equation*}
\frac{1}{\lambda_n}[\langle y_n - z, \nabla f(y_n) - \nabla f(x_n) \rangle]\leq g(x_n, z)-g(x_n, y_n),
\end{equation*}
for all $z \in C$. Now, substituting $n_m$  instead of $n$ in the above, we conclude that
\begin{equation}\label{ineq}
\frac{1}{\lambda_{n_m}}[\langle y_{n_m} - z, \nabla f(y_{n_m}) - \nabla f(x_{n_m}) \rangle]\leq g(x_{n_m}, z)-g(x_{n_m}, y_{n_m}),
\end{equation}
for all $z \in C$. From \eqref{ynxn}, \eqref{ineq}, Lemma \ref{2.1}, the conditions (iii), $A4$ and the fact that $x_{n_m}\rightharpoonup q$ as $m\rightarrow\infty$, we have that $g(q, z)\geq 0$ for all $z\in C$, then from $A2$, we have $g(z, q)\leq 0$, for all $z\in C$, i.e., $q\in EP(g)$.

Now, we show that $q\in F(Res^f_\varphi)$.
Since $\{u_n\}$ is bounded, hence $\{\nabla f(u_n)\}$ is bounded. Let $r_3=\displaystyle\sup_n\{\|\nabla f(u)\|, \|\nabla f(x_n)\|, \|\nabla f(u_n)\|, \|\nabla f(k_n)\|\}$. Then by \eqref{vf}, \eqref{dfvf}, \eqref{wnxn0}, \eqref{kn<xn}, \eqref{res4} and Lemmas \ref{2.5}, \ref{fsum<sumf},  we have
\begin{align}\label{unxnnorm}
  D_f(\hat{u}, x_{n+1})\leq&D_f(\hat{u}, h_n)\nonumber\\
  =& V_f(\hat{u}, \alpha_{n,1}\nabla f(u)+\alpha_{n,2}\nabla f(x_n)+\alpha_{n,3}\nabla f(u_n)+\alpha_{n,4} f(k_n))\nonumber\\
  =&  f(\hat{u}) - \langle\hat{u}, \alpha_{n,1}\nabla f(u)+\alpha_{n,2}\nabla f(x_n) +\alpha_{n,3}\nabla f(u_n)+ \alpha_{n,4}\nabla f(k_n) \rangle \nonumber\\
       &+ f^*(\alpha_{n,1}\nabla f(u)+\alpha_{n,2}\nabla f(x_n) +\alpha_{n,3}\nabla f(u_n)+\alpha_{n,3}\nabla f(k_n))\nonumber \\
       \leq &f(\hat{u})-\alpha_{n,1}\langle\hat{u}, \nabla f(u) \rangle- \alpha_{n,2}\langle \hat{u}, \nabla f(x_n)\rangle -\alpha_{n,3}\langle \hat{u}, \nabla f(u_n)\rangle\nonumber \\
       &-\alpha_{n,4}\langle \hat{u}, \nabla f(k_n)\rangle +\alpha_{n,1}f^*(\nabla f(u))+\alpha_{n,2}f^*(\nabla f(x_n))\nonumber\\
       &+\alpha_{n,3}f^*(\nabla f(u_n))+\alpha_{n,4}f^*(\nabla f(k_n))\nonumber\\
       &- \alpha_{n,2}\alpha_{n,3}\rho^*_{r_3}(\|\nabla f(x_n)-\nabla f(u_n)\|)\nonumber\\
       = &\alpha_{n,1}V_f( \hat{u}, \nabla f(u))+ \alpha_{n,2}V_f( \hat{u}, \nabla f(x_n))+\alpha_{n,3}V_f(\hat{u}, \nabla f(u_n))\nonumber\\
       &+\alpha_{n,4}V_f(\hat{u}, \nabla f(k_n))- \alpha_{n,2}\alpha_{n,3}\rho^*_{r_3}(\|\nabla f(x_n)-\nabla f(u_n)\|)\nonumber\\
       =&\alpha_{n,1}D_f( \hat{u}, u)+\alpha_{n,2}D_f( \hat{u}, x_n)+\alpha_{n,3}D_f(\hat{u}, u_n)+\alpha_{n,4}D_f(\hat{u}, k_n)\nonumber\\
       &- \gamma_{n,2}\gamma_{n,3}\rho_{r_1}^{*}(\|\nabla f(z_n) -\nabla f(Sz_n)\|)\nonumber\\
  \leq&\alpha_{n,1}D_f(\hat{u}, u)+(1-\alpha_{n,1})D_f(\hat{u}, x_n)\nonumber\\
  &- \alpha_{n,2}\alpha_{n,3}\rho^*_{r_3}(\|\nabla f(x_n)-\nabla f(u_n)\|).
\end{align}

Hence, it follows from the condition (i) and a similar technique as \eqref{tnstn1} that
\begin{equation}\label{unnorm}
  \displaystyle\lim_{n\rightarrow\infty}\|x_n-u_n\|=0.
\end{equation}
Note $\{x_n\}$, $\{z_n\}$ are bounded and $\nabla f$ is bounded on bounded subsets of $X$ and also
$\nabla f^*$ is bounded on bounded subsets of $X^*$. Then consequently $\{w_n\}$ is bounded. In a similar way, it follows that $\{k_n\}$ and $\{h_n\}$ are bounded.

From Lemma \ref{d_g}, the inequalities \eqref{Df},  \eqref{tnstn1} and \eqref{wntwn}, we have
\begin{equation}\label{tnwni}
          \displaystyle\lim_{n\rightarrow\infty}D_f(z_n, v_n)=0, \displaystyle\lim_{n\rightarrow\infty}D_f(z_n, Sz_n)=0, \displaystyle\lim_{n\rightarrow\infty}D_f(w_n, Sw_n)=0,
         \end{equation}
 also from \eqref{df<df}, we have
         \begin{equation*}\label{n}
            D_f(z_n, w_n)\leq \gamma_{n,1} D_f(z_n, v_n)+\gamma_{n,2} D_f(z_n, z_n)+\gamma_{n,3} D_f(z_n, Sz_n),
         \end{equation*}
so from \eqref{tnwni}, we obtain that
         \begin{equation*}\label{tnwn}
           \displaystyle\lim_{n\rightarrow\infty}D_f(z_n, w_n)=0.
         \end{equation*}
 Now, from Lemma \ref{2.2}, we conclude that
         \begin{equation}\label{tnwn1}
           \displaystyle\lim_{n\rightarrow\infty}\|z_n-w_n\|=0.
         \end{equation}
 By Lemma \ref{d_g}, the inequalities \eqref{total}, \eqref{unnorm} and \eqref{tnwn1}, we obtain that
 \begin{equation}\label{27}
  \displaystyle\lim_{n\rightarrow\infty} D_f(w_n, x_n)=0,\;\;\;\displaystyle\lim_{n\rightarrow\infty}D_f(w_n, u_n)=0.
 \end{equation}
Then, as the above, we have
\begin{equation*}\label{50}
   \displaystyle\lim_{n\rightarrow\infty}\|u_n-w_n\|=0.
\end{equation*}
 Therefore, from Lemma \ref{2.1}, it is implied that
\begin{equation}\label{50}
   \displaystyle\lim_{n\rightarrow\infty}\|\nabla f(u_n)-\nabla f(w_n)\|=0.
\end{equation}
 From \eqref{unnorm} and $x_{n_m}\rightharpoonup q$, we obtain that $u_{n_m}\rightharpoonup q$.
Since $u_n=Res^f_\varphi w_n$, then we have
\begin{equation*}
 \varphi(u_n, y) + \langle \nabla f(u_n)-\nabla f(w_n), y-u_n\rangle\geq 0,\;\;\; \forall\; y\in C.
\end{equation*}
Now, substituting $n_m$  instead of $n$ in the above, we conclude that
\begin{equation*}
 \varphi(u_{n_m}, y) + \langle \nabla f(u_{n_m})-\nabla f(w_{n_m}), y-u_{n_m}\rangle\geq 0,\;\;\; \forall\; y\in C.
\end{equation*}
By (B2), we obtain
\begin{equation}\label{29}
 \langle \nabla f(u_{n_m})-\nabla f(w_{n_m}), y-u_{n_m}\rangle\geq -\varphi(u_{n_m}, y)\geq \varphi(y, u_{n_m}),\;\;\; \forall\; y\in C.
\end{equation}
Now, we know from (B4) that $\varphi(x, .)$ is convex and lower semicontinuous. Then from \cite[Propositions 1.9.13 and 2.5.2]{arpod}, it is also weakly lower semicontinuous. Hence, letting $m\rightarrow\infty$, we conclude from \eqref{50} and \eqref{29}  that $\varphi(y, q)\leq 0$ for all $y\in C$. Then from \eqref{EP}, $q\in EP(\varphi)$. So by the condition (iii) of Lemma \ref{res3}, we have that $q\in F(Res^f_\varphi)$.
Therefore $q\in \Omega$.

It follows from \eqref{m}, \eqref{tnwni}, \eqref{27}, the boundedness of $\{ D_f(w_n, u)\}$ and the conditions (i), (iv) that
  $\lim_{n\rightarrow\infty}D_f(w_n, h_n)=0$.
Hence, from Lemma \ref{2.2} and the boundedness of $\{h_n\}$, we have
       \begin{equation*}\label{wnmn1}
         \displaystyle\lim_{n\rightarrow\infty}\|w_n-h_n\|=0.
       \end{equation*}
Then from \eqref{total}, \eqref{tnwn1} and the above, we obtain
   \begin{equation}\label{xnmn}
     \displaystyle\lim_{n\rightarrow\infty}\|x_n-h_n\|=0.
   \end{equation}
Now, we show that $\limsup_{n\rightarrow\infty}\langle h_n -\hat{u}, \nabla f( u)- \nabla f(\hat{u})\rangle\leq 0$. Since $\{x_n\}$ is bounded, we conclude that there exists a subsequence $\{x_{n_k}\}$ of $\{x_n\}$ which
\begin{equation*}\label{converg}
  \limsup_{n\rightarrow\infty}\langle x_n -\hat{u}, \nabla f(u)- \nabla f(\hat{u})\rangle  = \lim_{k\rightarrow\infty}\langle x_{n_k} -\hat{u}, \nabla f(u)- \nabla f(\hat{u})\rangle.
\end{equation*}
Since $\{x_{n_k}\}$ is bounded then there exists  a subsequence $\{x_{n_{k_{i}}}\}$ of $\{x_{n_k}\}$ that
converges weakly to some $\kappa\in \Omega$. Without loss of generality, we can assume that $x_{n_k}\rightharpoonup\kappa$.
Then from Lemma \ref{2.5}, we have that
\begin{equation}\label{converg}
  \limsup_{n\rightarrow\infty}\langle x_n -\hat{u}, \nabla f(u)- \nabla f(\hat{u})\rangle
  =  \langle \kappa -\hat{u}, \nabla f( u)- \nabla f(\hat{u})\rangle\leq 0.
\end{equation}
Next, It follows from \eqref{xnmn} and \eqref{converg} that
\begin{equation}\label{converg1}
  \limsup_{n\rightarrow\infty}\langle h_n -\hat{u}, \nabla f(u)- \nabla f(\hat{u})\rangle  = \limsup_{n\rightarrow\infty}\langle x_n -\hat{u}, \nabla f(u)- \nabla f(\hat{u})\rangle\leq 0.
\end{equation}
Therefore, from \eqref{u,xn+1}, \eqref{converg1} and Lemma \ref{an+1<an}, we have $\displaystyle\lim_{n\rightarrow\infty} D_f(\hat{u}, x_n)=0$. Therefore from Lemma \ref{2.2},
$x_n\rightarrow \hat{u}$ as $n\rightarrow\infty$.
\section*{Case2}
There exists a subsequence $\{n_j\}$ of $\{n\}$, such that
\begin{equation*}
  D_f (\hat{u}, x_{n_j}) < D_f (\hat{u}, x_{n_j+1}),
\end{equation*}
for all $j\in \mathbb{N}$. Hence, from Lemma \ref{2.18}, there exists a subsequence $\{m_k\} \subset \mathbb{N}$, such that $m_k \rightarrow \infty$  and the following properties are satisfied:
\begin{equation}\label{incres}
   D_f (\hat{u}, x_{m_k})\leq  D_f (\hat{u}, x_{m_{k}+1})\;\; \text{and} \;\; D_f (\hat{u}, x_k)\leq  D_f (\hat{u}, x_{m_{k}+1}),
\end{equation}
for all $k\in \mathbb{N}$.
From \eqref{l}, we have
         \begin{align*}
         (\alpha_{m_k,3}+\alpha_{m_k,4})&(\gamma_{m_k,2}+\gamma_{m_k,3})(1-\lambda_{m_k}c_1)D_f(y_{m_k}, x_{m_k}) \nonumber\\
           \leq& \alpha_{m_k,1} D_f(\hat{u}, u)+D_f( \hat{u}, x_{m_k})-D_f(\hat{u}, x_{{m_k}+1})\nonumber\\
          &-(\alpha_{m_k,3}+\alpha_{m_k,4})(\gamma_{m_k,2}+\gamma_{m_k,3})(1-\lambda_{m_k}c_2)D_f(z_{m_k}, y_{m_k}),
        \end{align*}
        for all $k\in \mathbb{N}$. Then using \eqref{incres}, the conditions (i), (ii) and (iii), it follows that $D_f(y_{m_k}, x_{m_k})\rightarrow 0$ as $k\rightarrow \infty$. So by Lemma \ref{2.2} and the boundedness of $\{x_{m_k}\}$, we obtain
      \begin{equation*}\label{ynxn1}
        \displaystyle\lim_{n\rightarrow\infty}\|y_{m_k} - x_{m_k}\|=0.
      \end{equation*}
      In a similar way, from \eqref{l}, we have
      \begin{equation*}\label{ynzntn1}
         \displaystyle\lim_{k\rightarrow\infty}\|y_{m_k} - z_{m_k}\|=0.
      \end{equation*}
Now proceeding with the same argument as in Case 1, we conclude that
\begin{align}\label{converg2}
  \limsup_{k\rightarrow\infty}\langle h_{m_k} -\hat{u},& \nabla f(u)- \nabla f(\hat{u})\rangle \nonumber\\
   &= \limsup_{k\rightarrow\infty}\langle x_{m_k} -\hat{u}, \nabla f(u)- \nabla f(\hat{u})\rangle\leq 0.
\end{align}
Also, from \eqref{u,xn+1} and \eqref{incres}, we obtain
\begin{align*}\label{u,xn+2}
   D_f(\hat{u}, x_{{m_k}+1}) & \nonumber \\
  \leq &(1- \alpha_{m_k,1})D_f(\hat{u}, x_{m_k})+\alpha_{m_k,1}\langle h_{m_k} -\hat{u}, \nabla f( u)- \nabla f(\hat{u})\rangle\nonumber\\
  \leq & (1- \alpha_{m_k,1})D_f(\hat{u}, x_{{m_k}+1})+\alpha_{m_k,1}\langle h_{m_k} -\hat{u}, \nabla f( u)- \nabla f(\hat{u})\rangle,
\end{align*}
hence,
\begin{equation*}\label{u,xn+2}
   D_f(\hat{u}, x_{{m_k}+1}) \leq\langle h_{m_k} -\hat{u}, \nabla f( u)- \nabla f(\hat{u})\rangle.
\end{equation*}
Then, it follows from \eqref{incres} that
\begin{equation}\label{uxk}
  D_f (\hat{u}, x_k)\leq  D_f (\hat{u}, x_{m_{k}+1}) \leq  \langle h_{m_k} -\hat{u}, \nabla f( u)- \nabla f(\hat{u})\rangle.
\end{equation}
From \eqref{converg2} and \eqref{uxk}, we conclude that
$\displaystyle\lim_{k\rightarrow\infty}D_f (\hat{u}, x_k) = 0$.
Therefore from Lemma \ref{2.2},
$x_k\rightarrow \hat{u}$ as $k\rightarrow\infty$. This completes the proof.
\end{proof}
\section{Applications and numerical examples}
The following theorem will be concluded from theorem \ref{maina}.
\begin{thm}\label{maina2}
Let $C$ be a nonempty closed convex subset of a real reflexive Banach
space $X$, and let $f : X \rightarrow \mathbb{R}$ be an admissible, strongly coercive Legendre function which
is bounded, uniformly Fr\'{e}chet differentiable and totally convex on bounded
subsets of $X$. Let  $g : C\times C \rightarrow \mathbb{R}$ be a bifunction satisfying
A1-A5. Assume that $S : C \rightarrow C$ is a Bregman nonexpansive mapping with $\hat{F}(S)=F(S)$. Let $\Omega =F(S)\cap EP(g) \neq\emptyset$. Suppose that $\{x_n\}$ is a sequence generated by $x_1\in C$, $u\in X$ and

    \begin{equation}\label{algo3}
    \left\{
    \begin{array}{lr}

       y_n=argmin\{\lambda_n g(x_n, y) +D_f(y, x_n):y\in C\},\\
       z_n= argmin\{\lambda_n g(y_n, y) +D_f(y, x_n):y\in C\},\\
      % t_n=argmin\{\lambda_n g(z_n, y) +D_f(y, x_n):y\in C\},\\
       %w_n=\nabla f^*(\gamma_{n,1}\nabla f(x_n)+\gamma_{n,2}\nabla f(z_n) +\gamma_{n,3}\nabla f(Sz_n)),\\
       v_n=\nabla f^*(\delta_n\nabla f(\overleftarrow{Proj}_C^f x_n)+(1-\delta_n)\nabla f(\overleftarrow{Proj}_C^f z_n)),\\
       w_n=\nabla f^*(\gamma_{n,1}\nabla f(v_n)+\gamma_{n,2}\nabla f(z_n) +\gamma_{n,3}\nabla f(Sz_n)),\\
       u_n=\overleftarrow{Proj}_C^fw_n,\\
   %    C_{n+1}=\{v\in C_n: D_f(v, u_n)\leq D_f(v, x_n)\},\\
 %      D_n=\{z\in C: \langle x_n-z, x_0 -x_n\rangle \geq 0\},\\
        k_n=\nabla f^*(\beta_n\nabla f(w_n)+(1-\beta_n)\nabla f(Sw_n)),\\
        h_n=\nabla f^* (\alpha_{n,1}\nabla f( u)+\alpha_{n,2}\nabla f(x_n)+\alpha_{n,3}\nabla f(u_n)+\alpha_{n,4}\nabla f(k_n)),\\
       x_{n+1}=\overleftarrow{Proj}_C^f h_n.
    \end{array} \right.
   \end{equation}
   where $\{\delta_n\}\subset(0, 1)$ and $\{\alpha_{n,i}\}^4_{i=1}$,  $\{\beta_n\}$ and $\{\lambda_n\}$ satisfy the following conditions:
     \begin{enumerate}
       \item [\emph{(i)}] $\{\alpha_{n,i}\}\subset (0, 1), \sum_{i=1}^{4}\alpha_{n,i}=1,  \lim_{n\rightarrow \infty}\alpha_{n,1}=0,$ $ \Sigma_{n=1}^{\infty}\alpha_{n,1}=\infty$, $\liminf_{n\rightarrow\infty} \alpha_{n,i}\alpha_{n,j}>0$ for all $i\neq j$ and $2\leq i,j\leq3$.
       \item [\emph{(ii)}] $\{\gamma_{n,i}\}\subset(0, 1)$, $\gamma_{n,1}+\gamma_{n,2}+\gamma_{n,3}=1$, $\liminf_{n\rightarrow\infty} \gamma_{n,i}\gamma_{n,j}>0$ for all $i\neq j$ and $1\leq i,j\leq3$.
       \item [\emph{(iii)}] $\{\lambda_n\}\subset [a, b]\subset (0, p)$, where $p=\min\{\frac{1}{c_1}, \frac{1}{c_2}\}$, $c_1, c_2$ are the Bregman-Lipschitz coefficients of $g$.
       \item [\emph{(iv)}] $\{\beta_n\}\subset (0, 1), \liminf_{n\rightarrow \infty}\beta_n(1-\beta_n)>0$.
     \end{enumerate}
     Then, the sequence $\{x_n\}$ converges strongly to $\overleftarrow{Proj}_{\Omega}^fu$.
\end{thm}
\begin{proof}
Putting $\varphi(x, y)=0$ for all $x, y\in C$ in Theorem \ref{maina}, by Lemma \ref{2.5}, we have $Res^f_\varphi=\overleftarrow{Proj}_C^f$, $u_n=\overleftarrow{Proj}_C^fw_n$ and $F(Res^f_\varphi)=C$.
\end{proof}

 \begin{ex}\label{breg}
Let $X = \mathbb{R}$, $C = [0, 2]$, $f(.) = \frac{1}{2}\|.\|^2$ and define the bifunction $g: C\times C\rightarrow \mathbb{R}$ by
     \begin{equation*}
       g(x,y):=16y^2+9xy-25x^2,
     \end{equation*}
for all $x$, $y \in C$. Next, $g$ satisfies  the conditions A1 - A6 as follows:\\
 (A1) $g$ is monotone:
 \begin{equation*}
   g(x, y) + g(y, x)=16y^2+9xy-25x^2+16x^2+9xy-25y^2=-9(x-y)^2\leq 0,
 \end{equation*}
  for all $x, y\in C$.\\
(A2) $g$ is pseudomonotone, for all $x$, $y\in C$, because $g$ is  monotone.\\
(A3) $g$ is Bregman-Lipschitz-type continuous  with $c_1 = c_2 = 9$:
\begin{align}\label{aaaa}
  g(x, y) + g(y, z) - g(x, z) &=16y^2 + 9xy-25x^2 \nonumber\\
  &+ 16z^2 + 9yz -25y^2- 16z^2 -9xz +25 x^2 \nonumber\\
  = &-9(y^2-xy-yz+xz)\nonumber\\
  =&-\frac{9}{2}(y^2-2xy+x^2+y^2-2yz+z^2-x^2+2xz-z^2)\nonumber\\
  =&-9D_f(x, y)- 9D_f(y, z)+9D_f(x, z)\nonumber\\
   \geq& -9D_f(x, y)- 9D_f(y, z).
  \end{align}
  (A4) Note since $C \subseteq   \mathbb{R}$ is weakly compact we need only show that $g: C \times C \to  \mathbb{R}$ is weakly sequentially continuous. Let $x$, $y\in C$ and let  $\{x_n\}$ and $\{y_n\}$ be two sequences in C converging weakly to x and y, respectively (note the weak and strong convergence in $\mathbb{R}$ are equivalent), and then
  \begin{equation*}
   \displaystyle\lim_{n\rightarrow\infty}g(x_n,y_n)=\displaystyle\lim_{n\rightarrow\infty} 16y_{n}^2+9x_ny_n-25x_{n}^2=16y^2+9xy-25x^2=g(x,y).
  \end{equation*}
  (A5)  Note   $g(x, .)$ is convex, lower semicontinuous and subdifferentiable on $C$ for
every fixed $x \in C$.\\
(A6) Note
\begin{align}\label{A6}
  \displaystyle&\limsup_{t\rightarrow 0}g(tx + (1 - t)y, z) =\nonumber\\
  &\displaystyle\limsup_{t\rightarrow 0}[16z^2+9(tx + (1 - t)y)z-25(tx + (1 - t)y)^2]\nonumber  \\
  &=g(y, z),
  \end{align}
for each $x$, $y$, $z \in C$.
Define $S:C\rightarrow C$ by $S(x)=\frac{x}{3}$, for all $x\in C$. Hence, $F(S)=\{0\}$ and
\begin{align*}
  D_f(Sx, Sy)&=D_f(\frac{x}{3}, \frac{y}{3})\\
  =& f(\frac{x}{3})-f(\frac{y}{3}) - \langle \frac{x}{3}-\frac{y}{3}, \frac{y}{3}\rangle\\
  =&\frac{x^2}{18}-\frac{y^2}{18}-\frac{xy}{9}+\frac{y^2}{9}\\
  =&\frac{1}{18}(x-y)^2\leq\frac{1}{2}(x-y)^2= D_f(x, y),
\end{align*}for all $x, y \in C$. Therefore $S$ is  Bregman nonexpansive. Let $p\in \hat{F(S)}$ then $C$ contains a sequence $\{x_n\}$ such that $x_n\rightharpoonup p$ and $\displaystyle\lim_{n\rightarrow\infty}(Sx_n - x_n))=0$, then $p=0$. So $\hat{F}(S)=\{0\}=F(S)$. Now, define the bifunction $\varphi: C\times C\rightarrow \mathbb{R}$ by $\varphi(x, y)=0$ for all $x, y\in C$. It is clear that $\varphi$ satisfies the conditions $B1-B4$. By Lemma \ref{2.5}, we conclude that $Res^f_\varphi=Proj^f_C$.
 Now, if $\lambda_n=\frac{1}{32}$ by definition of $y_n$ in our algorithm, we have
\begin{equation*}
   y_n=argmin\{\frac{1}{32} g(x_n, y) +D_f(y, x_n):y\in C\},
\end{equation*}
therefore, $y_n=\frac{23}{64} x_n$. Similarly
\begin{equation*}
  z_n= argmin\{\frac{1}{32} g(y_n, y) +D_f(y, x_n):y\in C\},
\end{equation*}
hence, $z_n=\frac{1841}{(64)^2}x_n$. Also $v_n=\delta_nP_Cx_n+(1-\delta_n)P_Cz_n$, $w_n=\gamma_{n,1}v_n+\gamma_{n,2}z_n+\frac{1}{3}\gamma_{n,3}z_n$, $u_n=P_Cw_n$, $k_n=\beta_nw_n+(1-\beta_n)\frac{1}{3}w_n=(\frac{1}{3}+\frac{2}{3}\beta_n)w_n$, $h_n=\alpha_{n,1}u+\alpha_{n,2}x_n+\alpha_{n,3}u_n+\alpha_{n,4}k_n$  and $x_{n+1}=P_Ch_n$.
We choose $\alpha_{n,1}=\frac{1}{4n}$, $\alpha_{n,2}=\alpha_{n,3}=\alpha_{n,4}=\frac{1}{3}-\frac{3}{4n}$, $\beta_n=\frac{1}{2}+\frac{1}{n}$, $\delta_n=\frac{1}{2}$ and $\gamma_{n,1}=\gamma_{n,2}=\gamma_{n,3}=\frac{1}{3}$. See the table \ref{tableexample1} and Figure \ref{pp1} with the initial point $x_1 = 5$ of the sequence
$\{x_n\}$.
 \end{ex}

\end{document}